\documentclass[11pt,oneside]{amsart}
\usepackage[mathscr]{eucal}
\usepackage{amssymb, amsmath, array, amscd}
\usepackage[dvips]{graphicx}
\usepackage{color}
\usepackage{epsfig}
\newtheorem{theorem}{Theorem}[section]
\newtheorem{corollary}[theorem]{Corollary}
\newtheorem{proposition}[theorem]{Proposition}
\newtheorem{lemma}[theorem]{Lemma}

\newtheorem{remark}[theorem]{Remark}

\begin{document}

\title[A characterization of Inoue surfaces]{A CHARACTERIZATION OF INOUE SURFACES}

\author{Marco Brunella}

\address{Marco Brunella, Institut de Math\'ematiques de Bourgogne
-- UMR 5584 -- 9 Avenue Savary, 21078 Dijon, France}

\begin{abstract}
We give a characterization of Inoue surfaces in terms of automorphic
pluriharmonic functions on a cyclic covering. Together with results
of Chiose and Toma, this completes the classification of compact
complex surfaces of K\"ahler rank one.
\end{abstract}

\maketitle

In this paper we shall prove a conjecture proposed in \cite{C-T}:

\begin{theorem}\label{main}
Let $S$ be a compact connected complex surface of algebraic
dimension 0. Suppose that there exists an infinite cyclic covering
$\widetilde S \buildrel\pi\over\rightarrow S$ (with covering
transformations generated by $\varphi\in Aut(\widetilde S)$) and a
nonconstant positive pluriharmonic function $F$ on $\widetilde S$
such that
$$F\circ\varphi = \lambda\cdot F$$
for some $\lambda\in{\mathbb R}^+$. Then $S$ is a (possibly blown
up) Inoue surface.
\end{theorem}

The class of Inoue surfaces was discovered by Inoue (and
independently Bombieri) around 1972 \cite{Ino} \cite{Nak}. They are
special (and explicit) compact quotients of ${\mathbb
H}\times{\mathbb C}$, and they enjoy the following properties:
\begin{enumerate}
\item[-] the first Betti number is 1, the second Betti number is 0;
\item[-] they admit holomorphic foliations;
\item[-] they do not contain compact complex curves.
\end{enumerate}
Conversely, Inoue proved in \cite{Ino} that any compact connected
complex surface with the above properties is a Inoue surface.

Our proof of Theorem \ref{main} will be ultimately a reduction to
Inoue's theorem. The pluriharmonic function $F$ naturally induces a
holomorphic (and possibly singular) foliation ${\mathcal F}$ on $S$.
By a ``topological'' study of such a foliation we will be able to
understand some topological structure of $S$, and in particular to
show that $c_2(S_{\rm min})=0$ or $c_1^2(S_{\rm min})=0$ (where
$S_{\rm min}$ denotes the minimal model of $S$). From this vanishing
of Chern numbers, and results of Kodaira and Inoue, the conclusion
will be immediate. Remark that, conversely and by construction,
every Inoue surface satisfies the hypotheses of Theorem \ref{main},
which therefore gives a precise characterization of Inoue surfaces.

Together with the results of \cite{C-T}, Theorem \ref{main} allows
to complete the classification of compact surfaces of K\"ahler rank
one. Recall \cite{H-L} \cite{C-T} that a compact connected complex
surface $S$ has {\it K\"ahler rank one} if it is not K\"ahlerian but
it admits a closed semipositive $(1,1)$-form, not identically
vanishing (this is not the original definition of \cite{H-L}, but it
is equivalent to it by the results of \cite{C-T}, see also
\cite{Lam} and \cite{Tom}).

\begin{corollary} \label{rankone} {\rm (\cite{C-T} + Theorem \ref{main})}.
The only compact connected complex surfaces of K\"ahler rank one
are:
\begin{enumerate}
\item[(1)] Non-K\"ahlerian elliptic fibrations;
\item[(2)] Certain Hopf surfaces, and their blow-ups;
\item[(3)] Inoue surfaces, and their blow-ups.
\end{enumerate}
\end{corollary}

In the case of a Inoue surface, a closed semipositive $(1,1)$-form
is given by $(dF/F)\wedge (d^cF/F)$, with $F$ as in Theorem
\ref{main}.

\section{Geometric preliminaries}

Let $S$ be a surface as in Theorem \ref{main}. Without loss of
generality, we may assume that $S$ is minimal, since the hypotheses
are clearly bimeromorphically invariant. The assumption $a(S)=0$
implies that $S$ belongs to the class ${\rm VII}_\circ$ \cite[p.
188]{BPV}, that is $b_1(S)=1$ and $kod(S)=-\infty$: the existence of
a positive nonconstant pluriharmonic function on some covering of
$S$ excludes the case of tori and K3 surfaces. For the same reason,
$S$ cannot be a Hopf surface.

We claim that, in order to prove Theorem \ref{main}, it is
sufficient to prove that
$$c_2(S)=0$$
or
$$c_1^2(S)=0.$$
Indeed, we firstly observe that these two conditions are equivalent,
by Noether formula and $\chi ({\mathcal O}_S)=0$ (which follows from
$S\in {\rm VII}_\circ$). Then, $c_2(S)=0$ and $b_1(S)=1$ imply
$b_2(S)=0$. By a classical result of Kodaira \cite[Th. 2.4]{Nak},
$S$ contains no compact complex curve, otherwise it would be a Hopf
surface. Since $S$ also admits a holomorphic foliation (see below),
all the hypotheses of Inoue's theorem \cite{Ino} are satisfied and
we get that $S$ is a Inoue surface.

The automorphic function $F$ on $\widetilde S$ induces a real
analytic map
$$f=\log F : S \longrightarrow {\mathbb S}^1={\mathbb R}\big/{\mathbb Z}
\cdot\log\lambda .$$ The regular fibers of $f$ are smooth Levi-flat
hypersurfaces in $S$, because $F$ is pluriharmonic. However, $f$
could have also some singular fibers, corresponding to critical
points of $F$. In fact, our aim is precisely to show that these
singular fibers do not exist at all, since this is clearly
equivalent to $c_2(S)=0$. In order to avoid some cumbersome
statement, we shall suppose that the fibers of $f$ are connected.
The general case requires only few straightforward modifications of
the proof below.

The holomorphic 1-form $\omega =\partial F \in\Omega^1(\widetilde
S)$ descends to $S$ to a holomorphic section (still denoted by
$\omega$) of $\Omega^1(S)\otimes L$, where $L$ is a flat line bundle
(the one defined by the cocycle $\lambda\in{\mathbb R}^+\subset
{\mathbb C}^* = H^1(S,{\mathbb C}^*)$). This twisted closed
holomorphic 1-form induces a holomorphic foliation ${\mathcal F}$ on
$S$, which is tangent to the fibers of $f$.

In the following it will be important to distinguish between the
singularities of ${\mathcal F}$, ${\rm Sing}({\mathcal F})$, and the
zeroes of $\omega$, ${\rm Z}(\omega )$. The former are only isolated
points, since (as customary) we like to deal with ``saturated''
foliations. The latter, on the contrary, may contain some compact
complex curves. Remark also that ${\rm Z}(\omega )$ coincides with
the set of critical points of $f$, ${\rm Crit}(f)$.

The foliation ${\mathcal F}$ has a normal bundle $N_{\mathcal F}$
and a tangent bundle $T_{\mathcal F}$ \cite{Br1}, which are related
to the canonical bundle $K_S$ of $S$ by the adjunction type relation
$$N_{\mathcal F}\otimes T_{\mathcal F} = K_S^{-1} .$$
Because ${\mathcal F}$ is generated by $\omega\in\Omega^1(S)\otimes
L$, we have
$$N_{\mathcal F} = L \otimes {\mathcal O}(-\sum m_jC_j)$$
where $\{ C_j\}$ are the curves contained in ${\rm Z}(\omega )$ (if
any) and $\{ m_j\}$ are the respective vanishing orders.

We shall prove below that ${\rm Z}(\omega )$ is at most composed by
isolated points, giving by the previous formula the flatness of
$N_{\mathcal F}=L$. Then we shall prove that either $c_2(S)=0$ or
$T_{\mathcal F}$ is also flat. But in this second case we therefore
get that $K_S$ is flat too, hence $c_1^2(S)=0$.

\section{The structure of the regular fibers}

Here we consider a regular fiber of $f$,
$$M_\vartheta = \{ f=\vartheta \}\ , \quad \vartheta\ {\rm regular\
value},$$ and prove that it has the expected structure.

\begin{proposition}\label{regular}
The leaves of ${\mathcal F}\vert_{M_\vartheta}$ are either all
isomorphic to ${\mathbb C}$, or all isomorphic to ${\mathbb C}^*$.
In the first case, $M_\vartheta$ is diffeomorphic to ${\mathbb
T}^3$, and ${\mathcal F}\vert_{M_\vartheta}$ is a linear totally
irrational foliation. In the second case, $M_\vartheta$ is a
${\mathbb S}^1$-bundle over ${\mathbb T}^2$, and ${\mathcal
F}\vert_{M_\vartheta}$ is the pull-back of a linear irrational
foliation on ${\mathbb T}^2$.
\end{proposition}

The first case will lead to Inoue surface of type $S_M$, and the
second case to those of type $S_{N,p,q,r;t}^{(+)}$ or
$S_{N,p,q,r}^{(-)}$ \cite{Ino}.

\begin{proof}
The foliation ${\mathcal F}_\vartheta = {\mathcal
F}\vert_{M_\vartheta}$ is defined by the closed and nonsingular
1-form $\beta = d^cF\vert_{M_\vartheta}$ (which is well defined on a
neighbourhood of any fiber, up to a multiplicative constant). We may
use some classical results of Tischler \cite[I.4]{God} concerning
the structure of (real) codimension one foliations defined by closed
1-forms. According to those results, the foliation can be smoothly
perturbed to a fiber bundle over the circle with fiber $\Sigma_g$,
the (real) oriented compact surface of genus $g\ge 1$. Note that,
since $a(S)=0$, the leaves of ${\mathcal F}_\vartheta$ cannot be all
compact, and so they are all dense in $M_\vartheta$. Moreover, by
using the flow of a smooth vector field $v$ on $M_\vartheta$ such
that $\beta (v)\equiv 1$, and the closedness of $\beta$, we see that
the leaves are all diffeomorphic to the same abelian covering of
$\Sigma_g$.

The above flow of $v$ sends leaves to leaves, but of course it does
not need to preserve the complex structure of the leaves, that is it
does not need to realize a conformal diffeomorphism between the
leaves. However, the compactness of $M_\vartheta$ implies, at least,
that such a diffeomorphism is quasi-conformal. In particular, all
the leaves have the same (conformal) universal covering: either they
are all parabolic, or all hyperbolic. For our purposes, it is
sufficient to prove that the leaves are parabolic: since they are
abelian coverings of $\Sigma_g$, this implies that $g=1$, and the
rest of the statement is standard \cite[I.4]{God}.

We can associate to ${\mathcal F}_\vartheta$ a closed positive
current $\Phi\in A^{1,1}(S)'$, by integration along the leaves
against the transverse measure defined by $\beta$ \cite{Ghy}: if
$\eta\in A^{1,1}(S)$, we define
$$\Phi (\eta )= \int_{M_\vartheta} \beta\wedge\eta .$$
Obviously this current does not charge compact complex curves, hence
by results of Lamari \cite{Lam} \cite[Rem. 8]{Tom} it is an {\it
exact} positive current. As a consequence of this, its De Rham
cohomology class $[\Phi ]$ has vanishing product with the Chern
class of $T_{\mathcal F}$:
$$c_1(T_{\mathcal F})\cdot [\Phi ]=0.$$

Let us show that this implies the parabolicity of the leaves (this
is a particularly simple instance of the foliated Gauss-Bonnet
theorem, see \cite{Ghy}). In the opposite case, we may put on the
leaves of ${\mathcal F}_\vartheta$ their Poincar\'e metric, which
can be seen as a hermitian metric on $T_{\mathcal
F}\vert_{M_\vartheta}$. It is a continuous metric \cite{Ghy}, and it
can be regularized by a smooth hermitian metric on $T_{\mathcal
F}\vert_{M_\vartheta}$ whose curvature along the leaves is still
strictly negative (for instance, with the help of the above vector
field $v$). We then extend this hermitian metric to the full
$T_{\mathcal F}$, on the full $S$, in any way. The curvature form
$\Theta\in A^{1,1}(S)$ clearly satisfies
$$\Phi (\Theta ) = \int_{M_\vartheta}\beta\wedge\Theta < 0.$$
This is in contradiction with the vanishing of $c_1(T_{\mathcal
F})\cdot [\Phi ]$.
\end{proof}

\begin{remark} {\rm
Let us stress a subtle detail of the previous proof. The current
$\Phi$ can be also considered as a current on the real threefold
$M_\vartheta$. As such, however, it is {\it not} exact. Thus, in
order to get the vanishing of $c_1(T_{\mathcal F})\cdot [\Phi ]$, we
used also the fact that the tangent bundle $T_{{\mathcal
F}_\vartheta}$ extends to the full $S$, or more precisely that its
Chern class in $H^2(M_\vartheta ,{\mathbb R})$ extends to $S$, which
is obvious in our case since we have a global foliation on $S$. Now,
one can imagine a more general situation, in which we have a
Levi-flat hypersurface $M$ in a class ${\rm VII}_\circ$ surface,
such that the Levi foliation is given by a closed 1-form (or, more
generally, admits a transverse measure invariant by holonomy). Is it
still true that the leaves of this Levi foliation are parabolic?}
\end{remark}

\section{The structure of the singularities}

In order to study ${\rm Sing}({\mathcal F})$ and ${\rm Z}(\omega )$,
we need a general lemma on critical points.

Let $\widetilde U$ be a smooth complex surface and let $D\subset
\widetilde U$ be a compact connected curve (with possibly several
irreducible components). Suppose that the intersection form on $D$
is negative definite, so that $D$ is contractible to one point
\cite[p. 72]{BPV}. After contraction, we get a normal surface $U$
and a point $q\in U$, image of $D$; we do not exclude that $q$ is a
smooth point. Let now $\widetilde H$ be a holomorphic function on
$\widetilde U$, vanishing on $D$, such that
$${\rm Crit}(\widetilde H)=D.$$
After contraction, we thus get a holomorphic function $H$ on $U$
with (at most) an isolated critical point at $q$. If $B$ is a small
ball centered at $q$, then $H_0 = \{ H=0\}\cap B$ is a collection of
$k$ discs $H_0^1,\ldots ,H_0^k$ passing through $q$, whereas
$H_\varepsilon = \{ H=\varepsilon\}\cap B$ ($\varepsilon$ small and
not zero) is a connected curve with $k$ boundary components. The
topological type of $H_\varepsilon$ does not depend on $\varepsilon$
(small and not zero), it is the so-called Milnor fiber of $H$ at
$q$.

\begin{lemma} \label{genuszero}
Under the previous notation, suppose that the genus of the Milnor
fiber of $H$ at $q$ is zero. Then:
\begin{enumerate}
\item[(1)] $q$ is a smooth point of $U$;
\item[(2)] either $q$ is a regular point for $H$, or it is a critical
point of Morse type.
\end{enumerate}
\end{lemma}

\begin{proof}
The hypothesis means that the Milnor fiber is a sphere with $k$
holes. By a standard construction, we may glue to
$W=\cup_{|\varepsilon | < r} H_\varepsilon$ ($r>0$ small) a
collection of $k$ bidiscs in such a way that we obtain a normal
complex surface $V$ and a proper holomorphic map
$$G : V\longrightarrow {\mathbb D}(r)$$
such that:
\begin{enumerate}
\item[(i)] $W\subset V$ and $G\vert_W=H$;
\item[(ii)] $G_\varepsilon = G^{-1}(\varepsilon )$ is a smooth
rational curve for every nonzero $\varepsilon\in{\mathbb D}(r)$;
\item[(iii)] $G_0 = G^{-1}(0)$ is a collection of $k$ rational
curves $G_0^1,\ldots ,G_0^k$ passing through $q$, with $G_0^j\cap W
=H_0^j$ for every $j$.
\end{enumerate}

Remark that all the components $G_0^j$ of $G_0$ have multiplicity
$1$, i.e. $G$ vanishes along $G_0^j\setminus\{ q\}$ at first order
only. On the other hand, we may blow-up $q$ to the original $D$, and
we get a smooth complex surface $\widetilde V$ and a map
$$\widetilde G : \widetilde V \longrightarrow {\mathbb D}(r)$$
whose fiber over $0$ is $\widehat G_0^1 \cup\ldots\cup\widehat G_0^k
\cup D$, with $\widehat G_0^j$ the strict transform of $G_0^j$. By
construction, we have $${\rm mult}(\widehat G_0^j)=1$$ for every $j$
and $${\rm mult}(C)\ge 2$$ for every irreducible component $C$ of
$D$, since ${\rm Crit}(\widetilde H)=D$.

Recall now \cite[p. 142]{BPV} that such a $\widetilde V$ can be also
blow-down to the trivial fibration ${\mathbb D}(r)\times{\mathbb
C}P^1$, in such a way that the singular fiber of $\widetilde G$ is
sent to the regular fiber $\{ 0\}\times{\mathbb C}P^1$. In other
words, that singular fiber is obtained from a regular fiber by a
sequence of monoidal transformations. It is then easy to see that
$D$ necessarily contains a $(-1)$-curve: the reason is that a
monoidal transformation at a point belonging to an irreducible
component of multiplicity $m$ creates a new irreducible component
whose multiplicity will be not less than $m$. By iterating this
principle, we see that $D$ contracts to a regular point, whence the
first part of the lemma.

Moreover, after this contraction the singular fiber becomes a curve
(the fiber $G_0$ in the now smooth surface $V$) still dominating a
regular fiber, hence in particular it has only normal crossings.
Since all the components of $G_0$ pass through $q$, we get $k=1$
($G_0$ is a single smooth rational curve of selfintersection $0$) or
$k=2$ ($G_0$ is a pair of two smooth rational curves of
selfintersection $-1$). In the first case $q$ is a regular point,
and in the second case it is a Morse type critical point.
\end{proof}

\begin{remark} {\rm
If $H:{\mathbb C}^2\to {\mathbb C}$ has an isolated critical point
whose Milnor fiber has genus zero, the the critical point is of
Morse type: it is a particular case of the previous lemma, but it is
also a consequence of classical formulae estimating the genus of the
Milnor fiber. However, some care is needed when ${\mathbb C}^2$ is
replaced by a singular surface. For instance, take the function $zw$
on ${\mathbb C}^2$ and quotient by the involution $(z,w)\mapsto
(-z,-w)$. We get a normal surface $U$ and a holomorphic function $H$
on $U$ with an isolated critical point whose Milnor fiber has genus
zero. This kind of examples (and more complicated ones) do not
appear in Lemma \ref{genuszero} because, when we take the resolution
$\widetilde U\to U$, the critical set of $\widetilde H$ is not the
{\it full} exceptional divisor $D$.}
\end{remark}

We can now return to our compact complex surface $S$.

\begin{proposition}\label{singular}
The zero set ${\rm Z}(\omega )$ is composed only by isolated points,
all of Morse type. In particular, the normal bundle $N_{\mathcal F}$
coincides with the flat line bundle $L$.
\end{proposition}

\begin{proof}
Let $D$ be a connected component of ${\rm Z}(\omega )$. If it is a
curve, then it is a tree of rational curves with negative definite
intersection form: this follows from results of Nakamura on the
possible configurations of curves on ${\rm VII}_\circ$ surfaces
\cite{Nak}, and the absence of elliptic curves and cycles of
rational curves \cite{Tom} \cite{C-T}. In particular, $D$ is simply
connected, and so the (twisted) closed 1-form $\omega$ is exact on a
neighbourhood $\widetilde U$ of $D$: $\omega = d\widetilde H$ and
${\rm Crit}(\widetilde H)=D$. We therefore are in the setting of
Lemma \ref{genuszero}, and we have just to verify the genus zero
hypothesis.

Now, $D$ is contained in a singular fiber $M_{\vartheta_0}$, which
can be approximated by regular ones, on which we already know that
the foliation has leaves ${\mathbb C}$ or ${\mathbb C}^*$. It
follows obviously that the Milnor fiber has genus zero, and so by
Lemma \ref{genuszero} the contraction of $D$ produces a smooth
point. But we are also assuming since the beginning that $S$ is
minimal, hence such a contraction cannot exist and so ${\rm
Z}(\omega )$ is composed only by isolated points. By a similar
argument, and again Lemma \ref{genuszero}, all these points are of
Morse type.
\end{proof}

\section{The planar case}

Let ${\rm Cv}(f)$ denote the critical values of $f: S\to {\mathbb
S}^1$, and let $J$ be a connected component of ${\mathbb S}^1
\setminus{\rm Cv}(f)$. On $f^{-1}(J)$, the foliation ${\mathcal F}$
is nonsingular, and it is given by the transverse intersection of
the Kernels of the closed 1-form $df=dF/F$ and the integrable 1-form
$d^cF/F$. It follows that the differentiable type of ${\mathcal
F}\vert_{M_\vartheta}$ does not depend on $\vartheta\in J$. We shall
say that $J$ is of type ${\mathbb C}$ (resp. type ${\mathbb C}^*$)
if the leaves of ${\mathcal F}$ on $f^{-1}(J)$ are isomorphic to
${\mathbb C}$ (resp. ${\mathbb C}^*$), according with Proposition
\ref{regular}. In this section we shall prove that the existence of
a component $J$ of type ${\mathbb C}$ implies that $S$ is a Inoue
surface of type $S_M$.

Let us firstly recast Proposition \ref{singular} in the context of
uniformisation of foliations \cite{Br2}. Recall that the leaves of a
singular foliation are {\it not} simply the leaves outside the
singular points: generally speaking, and in order to get a workable
definition, we need to compactify some separatrices, the so-called
vanishing ends \cite[p. 732]{Br2}. However, the presence of a
vanishing end implies the existence of a rational curve on the
surface, invariant by the foliation, and over which the tangent
bundle of the foliation has strictly positive degree \cite[p.
733]{Br2}. We claim that this cannot happen in our context, and so
the leaves of our ${\mathcal F}$ are just equal to the leaves of
${\mathcal F}\vert_{S\setminus{\rm Sing}({\mathcal F})}$.

Indeed, if $C\subset S$ is a rational curve, invariant by ${\mathcal
F}$, then $c_1(T_{\mathcal F})\cdot C = 2 - {\rm Z}({\mathcal
F},C)$, where ${\rm Z}({\mathcal F},C)$ is the number of the
singularities of ${\mathcal F}$ along $C$ \cite{Br1}. On the other
hand, since these singularities are all of Morse type (${\rm CS}$
residue equal to $-1$), we also have $C^2 = -{\rm Z}({\mathcal
F},C)$. Hence
$$c_1(T_{\mathcal F})\cdot C = 2 + C^2\le 0$$
since, by minimality, $C^2\le -2$.

We shall also use the main result of \cite{Br2}, which says that,
since $S$ is not a Hopf surface nor a Kato surface, the foliation
${\mathcal F}$ is {\it uniformisable}, i.e. does not have vanishing
cycles (the reader is invited to give a simple proof of this result,
in our very special context).

\begin{proposition} \label{planar}
Let $J\subset{\mathbb S}^1\setminus{\rm Cv}(f)$ be a component of
type ${\mathbb C}$. Then $J={\mathbb S}^1$, and $S$ is a Inoue
surface of type $S_M$.
\end{proposition}

\begin{proof}
Suppose, by contradiction, that $J\not= {\mathbb S}^1$, and let
$M_{\vartheta_0}$ be a fiber in the boundary of $f^{-1}(J)$. Such a
fiber must contain a singular point $p\in {\rm Sing}({\mathcal F})$.
Let ${\mathcal L}$ be the leaf corresponding to one of the two
separatrices of $p$, and let $\gamma\subset{\mathcal L}$ be a cycle
close to $p$ and turning around it. This cycle, which has no
holonomy, can be slightly deformed to a cycle $\gamma
'\subset{\mathcal L}'$, where ${\mathcal L}'$ is a leaf contained in
$M_\vartheta$, $\vartheta\in J$ close to $\vartheta_0$. But
${\mathcal L}'$ is simply connected, hence $\gamma '$ is homotopic
to zero in ${\mathcal L}'$ and so $\gamma$ is homotopic to zero in
${\mathcal L}$, by the absence of vanishing cycles.

It follows that ${\mathcal L}$ is isomorphic to ${\mathbb C}$, and
${\mathcal L}\cup\{ p\}$ is a rational curve $C$ which contains a
unique singularity of the foliation. By the formulae above,
$C^2=-1$, contradicting the minimality of $S$.

Therefore $J={\mathbb S}^1$, and so $S$ is a ${\mathbb T}^3$-bundle
over the circle, without singularities. By \cite{Ino}, it is a Inoue
surface of type $S_M$.
\end{proof}

\begin{remark}\label{stein} {\rm
It is worth observing that, in our quite special context, the proof
of Inoue's theorem can be highly simplified. Indeed, and by our
previous results, on the universal covering $\widehat S$ of $S$ the
foliation is given by a submersion $\pi :\widehat S \to {\mathbb H}$
(with ${\rm Im}(\pi )$ coming from the pluriharmonic function $F$ on
$\widetilde S$) all of whose fibers are isomorphic to ${\mathbb C}$.
The key point is to prove that such a universal covering is a
product:
$$\widehat S = {\mathbb H}\times{\mathbb C} .$$
Indeed, once we know this fact, it remains to study the action of
$\Gamma = \pi_1(S)$ on ${\mathbb H}\times{\mathbb C}$. But we
already know a lot of properties of such an action (for instance,
$\Gamma$ is a semidirect product of ${\mathbb Z}$ and ${\mathbb
Z}^3$, which acts on the ${\mathbb H}$-factor in a special affine
way, etc.), and using that knowledge it is easy to conclude that $S$
is a Inoue surface of type $S_M$.

In order to prove that $\widehat S$ is a product, it is sufficient
to show that $\pi$ is a locally trivial fibration, i.e. that every
$z\in{\mathbb H}$ has a neighbourhood $U_z$ such that $\pi^{-1}(U_z)
= U_z\times{\mathbb C}$. By a classical theorem of Nishino
\cite{Nis}, this is equivalent to show that $V_z = \pi^{-1}(U_z)$ is
Stein. By an argument of Ohsawa \cite{Ohs}, the Steinness of $V_z$
follows from the existence of a smooth (not holomorphic!) foliation
${\mathcal H}$ on $V_z$ whose leaves are holomorphic sections of
$\pi$ over $U_z$ (i.e., $V_z$ is trivialisable by a smooth foliation
with holomorphic leaves).

Now, in our case such a foliation ${\mathcal H}$ is easy to
construct. On any fiber $M_\vartheta$ we can take a real analytic
foliation by real curves transverse to ${\mathcal
F}\vert_{M_\vartheta}$. By complexifying, we get, on a neighbourhood
of $M_\vartheta$, a real analytic foliation by complex curves,
transverse to ${\mathcal F}$. Using the special form of ${\mathcal
F}$, it is easy to see that this foliation, lifted to $\widehat S$,
as the required property (here $U_z$ is an horizontal strip in
${\mathbb H}$).   }
\end{remark}

\section{The cylindrical case}

Assume now that every $J\subset{\mathbb S}^1\setminus{\rm Cv}(f)$ is
of type ${\mathbb C}^*$.

\begin{lemma}\label{leaves}
Every leaf of ${\mathcal F}$ is isomorphic to ${\mathbb C}^*$.
\end{lemma}

\begin{proof}
The same argument used in Proposition \ref{planar} shows that, if
${\mathcal L}$ is a leaf in a singular fiber $M_{\vartheta_0}$
corresponding to a separatrix of some singular point $p$, then
${\mathcal L}$ must be diffeomorphic to the cylinder ${\mathbb
R}\times{\mathbb S}^1$. Indeed, ${\mathcal L}$ cannot be simply
connected by the minimality of $S$, and cannot have a fundamental
group larger than ${\mathbb Z}$ by the absence of vanishing cycles
(and of holonomy).

Remark now that, since the foliation is defined by a closed 1-form,
every leaf in $M_{\vartheta_0}$ is either dense or properly embedded
in the complement of the singularities: this follows from the fact
that the holonomy pseudogroup of the foliation is composed only by
translations on ${\mathbb R}$. The second possibility occurs only
when the leaf is a ``double separatrix'', i.e. a cylinder with both
ends converging to singular points, in which case the leaf is
isomorphic to ${\mathbb C}^*$ and its closure is a rational curve of
selfintersection $-2$.

Consider now an arbitrary dense leaf ${\mathcal L}'$ in
$M_{\vartheta_0}$. By the absence of vanishing cycles (and of
holonomy), we get again that its fundamental group cannot be larger
than ${\mathbb Z}$. On the other hand, this leaf accumulates to the
separatrix ${\mathcal L}$, hence it cannot simply connected
otherwise ${\mathcal L}$ would be simply connected too. In
conclusion, we see that {\it every} leaf in $M_{\vartheta_0}$ is a
cylinder.

To find the conformal type of the leaves, observe that an end of a
leaf is either convergent to a singular point (in which case it is
obviously parabolic), or it intersects a small ball $B$ around a
singular point $p$ along infinitely many annuli. These annuli are
not homotopic to zero, by the previous considerations. Moreover, we
can extract among them infinitely many ones with bounded modulus
(i.e., all isomorphic to $\{ r<|z|<1\}$ with $r$ varying in a
compact subset of $(0,1)$). It follows that the end is parabolic,
and the leaf is isomorphic to ${\mathbb C}^*$.
\end{proof}

\begin{lemma}\label{bisection}
The line bundle $T_{\mathcal F}^{\otimes 2}$ admits a continuous
section on $S\setminus {\rm Sing}({\mathcal F})$ which is nowhere
vanishing.
\end{lemma}

\begin{proof}
The complex curve ${\mathbb C}^*$ admits a ``almost canonical''
holomorphic vector field: the vector field
$z\frac{\partial}{\partial z}$, which can be almost uniquely
characterized as a complete holomorphic vector field whose flow is
$2\pi {\rm i}$-periodic. There is however a minor ambiguity, since
also the vector field $-z\frac{\partial}{\partial z}$ (conjugate to
the previous one by the inversion $z\mapsto 1/z$, which exchanges
the two ends) is complete and $2\pi {\rm i}$-periodic. This
ambiguity can be removed when we take the square:
$(z\frac{\partial}{\partial z})^{\otimes 2} =
(-z\frac{\partial}{\partial z})^{\otimes 2}$. This means that, given
any foliation ${\mathcal F}$ with leaves isomorphic to ${\mathbb
C}^*$, we get a {\it canonical} nonvanishing section of $T_{\mathcal
F}^{\otimes 2}$ on $S\setminus {\rm Sing}({\mathcal F})$, by the
previous recipe. The point to be proved is that such a section is
(at least) continuous.

This is equivalent to prove the following. Let $T\subset S$ be a
local transversal to ${\mathcal F}$, isomorphic to a disc, and let
$V_T$ be the corresponding holonomy tube \cite[p. 734]{Br2}. It is a
complex surface, equipped with a submersion $Q_T : V_T\to T$, all of
whose fibers are isomorphic to ${\mathbb C}^*$, and a section
$q_T:T\to V_T$. For every $t\in T$ we have a unique isomorphism
$i_t$ from $Q_T^{-1}(t)$ to ${\mathbb C}^*$, sending $q_T(t)$ to $1$
(really, there is again a ${\mathbb Z}_2$-ambiguity, which however
can be easily removed by prescribing an homotopy class). Therefore
we get a trivialising map
$$u : V_T \longrightarrow T\times{\mathbb C}^*$$
$$u\vert_{Q_T^{-1}(t)} = (t,i_t)$$
and the continuity of the above canonical section of $T_{\mathcal
F}^{\otimes 2}$ is clearly equivalent to the continuity of $u$ (for
every transversal $T$).

As shown in \cite[p. 78]{Ghy} (see also \cite[I.2]{Nis}), the
continuity of $u$ readily follows from Koebe's Theorem. Let us
recall the argument, for completeness.

Take a compact $K\subset Q_T^{-1}(t_0)$ and an exhaustion of
$Q_T^{-1}(t_0)$ by relatively compact open subsets $\{
\Omega_n\}_{n\in{\mathbb N}}$. By a standard argument (e.g. Royden's
Lemma), each $\Omega_n$ can be holomorphically deformed to the
nearby fibers $Q_T^{-1}(t)$, $t\in U_n$ $=$ a neighbourhood of $t_0$
in $T$. Thus, the maps $i_t$, $t\in U_n$, can be seen as all defined
on the same domain $\Omega_n$. By Koebe's Theorem, the distorsion of
$i_t$ on $K\subset\Omega_n$ is uniformly bounded by a constant which
tends to zero as $n\to\infty$, since $Q_T^{-1}(t_0)$ is parabolic.
We get in this way that $i_t\vert_K$ uniformly converge to
$i_{t_0}\vert_K$ as $t\to t_0$, and since $K$ was arbitrary we get
the continuity of $u$.
\end{proof}

Remark that, a posteriori, the above map $u$ will be even
holomorphic, as well as the canonical section of $T_{\mathcal
F}^{\otimes 2}$.

It is now easy to complete the proof of Theorem \ref{main}.

\begin{proposition}\label{cylindrical}
If every component of ${\mathbb S}^1\setminus{\rm Cv}(f)$ is of type
${\mathbb C}^*$, then there is only one component, equal to
${\mathbb S}^1$, and $S$ is a Inoue surface of type
$S_{N,p,q,r;t}^{(+)}$ or $S_{N,p,q,r}^{(-)}$.
\end{proposition}

\begin{proof}
By Lemma \ref{bisection}, the line bundle $T_{\mathcal F}^{\otimes
2}$ is topologically trivial, i.e. it is flat. From Proposition
\ref{singular} and $K_S^{-1} = N_{\mathcal F}\otimes T_{\mathcal F}$
it follows that $K_S$ is flat too, and so $c_1^2(S)=0$. As explained
at the beginning, this is the same as $c_2(S)=0$, the foliation is
nonsingular, and $S$ is a Inoue surface (of the claimed type).
\end{proof}

As in the planar case, also in the cylindrical case we do not need
the full strength of Inoue's theorem, since we can directly prove
that a covering of $S$ is isomorphic to ${\mathbb H}\times{\mathbb
C}^*$.


\begin{thebibliography}{20}
\bibitem[BPV]{BPV} W. Barth, C. Peters, A. Van de Ven, {\sl Compact
complex surfaces}, Ergebnisse der Mathematik (3)4, Springer (1984)
\bibitem[Br1]{Br1} M. Brunella, {\sl Foliations on complex projective
surfaces}, in {\sl Dynamical systems. Part II}, Pubbl. Scuola Norm.
Sup. Pisa (2003), 49--77
\bibitem[Br2]{Br2} M. Brunella, {\sl Nonuniformisable foliations on
compact complex surfaces}, Mosc. Math. J. 9 (2009), 729–-748
\bibitem[C-T]{C-T} I. Chiose, M. Toma, {\sl On compact complex
surfaces of K\"ahler rank one}, preprint arXiv (2010)
\bibitem[Ghy]{Ghy} \'E. Ghys, {\sl Laminations par surfaces de
Riemann}, in {\sl Dynamique et g\'eom\'etrie complexes}, Panor. et
Synth. 8, Soc. Math. France (1999), 49--95
\bibitem[God]{God} C. Godbillon, {\sl Feuilletages. \'Etudes
g\'eom\'etriques}, Progress in Math. 98, Birkh\"auser (1991)
\bibitem[H-L]{H-L} R. Harvey, H. B. Lawson, {\sl An intrinsic
characterization of K\"ahler manifolds}, Invent. Math. 74 (1983),
169-–198
\bibitem[Ino]{Ino} M. Inoue, {\sl On surfaces of Class
${\rm VII}_\circ$}, Invent. Math. 24 (1974), 269-–310
\bibitem[Lam]{Lam} A. Lamari, {\sl Le c\^one k\"ahl\'erien d'une
surface}, J. Math. Pures Appl. 78 (1999), 249–-263
\bibitem[Nak]{Nak} I. Nakamura, {\sl Towards classification of
non-K\"ahlerian complex surfaces}, Sugaku Expositions 2 (1989),
209--229
\bibitem[Nis]{Nis} T. Nishino, {\sl Nouvelles recherches sur les
fonctions enti\`eres de plusieurs variables complexes. II. Fonctions
enti\`eres qui se r\'eduisent \`a celles d'une variable}, J. Math.
Kyoto Univ. 9 (1969), 221-–274
\bibitem[Ohs]{Ohs} T. Ohsawa, {\sl A note on the variation of Riemann
surfaces}, Nagoya Math. J. 142 (1996), 1-–4
\bibitem[Tom]{Tom} M. Toma, {\sl On the K\"ahler rank of compact
complex surfaces}, Bull. Soc. Math. France 136 (2008), 243–-260
\end{thebibliography}
\end{document}